\documentclass[a4paper,11pt,leqno]{amsart}

\usepackage{amsmath}
\usepackage{amsthm}
\usepackage{amsfonts}
\usepackage{amssymb}

\usepackage{mathrsfs}

\usepackage[all]{xy}
\usepackage{colonequals}
\usepackage{color}
\usepackage[dvipsnames]{xcolor}
\usepackage{tikz}
\usepackage{enumerate}
\usetikzlibrary{calc}

\newtheorem{theorem}{Theorem}[section]{\bf}{\it}
\newtheorem*{theorem*}{Theorem}{\bf}{\it}
\newtheorem{lemma}[theorem]{Lemma}{\bf}{\it}
\newtheorem*{lemma*}{Lemma}{\bf}{\it}
{\bf}{\it}
\newtheorem*{fact*}{Fact}{\bf}{\it}
\newtheorem{proposition}[theorem]{Proposition}{\bf}{\it}
{\bf}{\it}

\theoremstyle{remark}
\newtheorem{remark}[theorem]{Remark}
\newtheorem*{remark*}{Remark}

\newtheorem{definition}[theorem]{Definition}

\numberwithin{equation}{section}

\newcommand{\C}{\mathbb{C}}

\newcommand{\N}{\mathbb{N}}
\newcommand{\Z}{\mathbb{Z}}
\newcommand{\bS}{\mathbb{S}}
\newcommand{\bD}{\mathbb{D}}

\newcommand{\dist}{\mathrm{dist}}

\newcommand{\eps}{\varepsilon}

\newcommand{\diam}{\operatorname{diam}}
\newcommand{\inv}{^{-1}}

\renewcommand{\emptyset}{\varnothing}

\begin{document}

\title{Sto\"ilow's theorem revisited}
\date{\today}

\author{Rami Luisto}
\address{Department of Mathematics and Statistics, P.O. Box 35, FI-40014 University of Jyv\"askyl\"a, Finland \and
Department of Mathematical Analysis, Sokolovska 83, Praha 8, 186 75, Charles University in Prague}
\email{rami.luisto@gmail.com}

\thanks{
R.L. has been partially supported by a grant of the Finnish Academy of Science and Letters,
  the Academy of Finland
  (grant 288501 `\emph{Geometry of subRiemannian groups}')
  and by the European Research Council
  (ERC Starting Grant 713998 GeoMeG `\emph{Geometry of Metric Groups}').
}

\author{Pekka Pankka}
\address{Department of Mathematics and Statistics, P.O. Box 68 (Gustaf H\"allstr\"omin katu 2b), FI-00014 University of Helsinki, Finland}
\email{pekka.pankka@helsinki.fi}

\thanks{P.P. has been partially supported by the Academy of Finland projects \#256228 and \#297258.}

\begin{abstract}
  Sto\"ilow's theorem from 1928 states that a continuous, open, and light map between surfaces is a discrete map with a discrete branch set. This result implies that such maps between orientable surfaces are locally modeled by power maps $z\mapsto z^k$ and admit a holomorphic factorization. 

  The purpose of this expository article is to give a proof of this classical theorem having readers in mind
  that are interested in continuous, open and discrete maps.
\end{abstract}


\maketitle

\section{Introduction}

Sto\"ilow's classical theorem in \cite{Stoilow} states that \emph{a continuous, open and light map between surfaces is a discrete map which has a discrete branch set}. In what follows, we call this theorem \emph{Sto\"ilow's discreteness theorem}.

Recall that a continuous map $f\colon X\to Y$ between topological spaces is \emph{light} if the pre-image $f^{-1}(y)$ of each point $y\in Y$ is totally disconnected, and \emph{discrete} if $f^{-1}(y)$ is a discrete subset of $X$. A continuous map is \emph{open} if the image of each open set is an open set. The \emph{branch set} $B_f$ of a continuous map $f\colon X\to Y$ is the set of points $x\in X$ at which $f$ fails to be a local homeomorphism.

In \cite{Stoilow} Sto\"ilow shows that these maps are locally modeled by power maps $z\mapsto z^k$; see \cite[p.\,372]{Stoilow} for the discussion.
This local description indicates a deep connection between continuous, open and light maps and holomorphic maps between surfaces. This connection was found by Sto\"ilow \cite[p.\,120]{Stoilow-Book-Old} and Whyburn \cite[Theorem X.5.1, p.\,198]{Whyburn-AT} and \cite[p.\,103]{Whyburn-TA}: 
\emph{For a continuous, open and light map $f\colon S\to S'$ between orientable Riemann surfaces there exists a Riemann surface $\tilde S$ and a homeomorphism $h \colon S\to \tilde S$ for which $f \circ h^{-1} \colon \tilde S \to S'$ is a holomorphic map;} the Riemann surface $\tilde S$ in this statement is naturally the Riemann surface associated to the map $f$. The first edition of \cite{Whyburn-TA}, published in 1956, does not give this result a specific name, but already in the second edition from 1964 the result is referred as \emph{Sto\"ilow's theorem}.

In this expository article we discuss the proof of Sto\"ilow's discreteness theorem having readers in mind that are interested in discrete and open maps, such as quasiregular maps (see e.g.\;\cite{RickmanBook}) or Thurston maps (see e.g.\;\cite{Bonk-Meyer-book}).
For this reason we separate Sto\"ilow's theorem into two parts: the discreteness of the map and the discreteness of the branch set.

\begin{theorem}[Sto\"ilow, 1928]
  \label{thm:light_to_discrete}
  Let $\Omega$ be a domain in $\C$ and $f\colon \Omega \to \C$ a continuous, open and light map. Then $f$ is a discrete map.
\end{theorem}

\begin{theorem}[Sto\"ilow, 1928]
  \label{thm:discrete_branch}
  Let $\Omega$ be a domain in $\C$ and $f\colon \Omega \to \C$ a continuous, open and discrete map. Then $B_f$ is a discrete set.
\end{theorem}

It is interesting to notice that both results stem from path-lifting arguments. Indeed, after path-lifting results are established, standard applications of the Jordan curve theorem yield both Theorems \ref{thm:light_to_discrete} and \ref{thm:discrete_branch}. For continuous, open and discrete maps, we may use Rickman's path-lifting theorem \cite{Rickman-Duke} and for continuous, open and light maps the method of Floyd \cite{Floyd}. Floyd's method suffices for all our purposes and we recall it in Section \ref{sec:path-lifting}; see \cite{Luisto-Characterization} for a more detailed discussion on path-lifting methods.

\bigskip

Having Theorems \ref{thm:light_to_discrete} and \ref{thm:discrete_branch} at our disposal,
it is a straightforward covering-space argument to show that continuous, open and light maps between surfaces are locally modeled by power maps. 
\begin{theorem}[Sto\"ilow, 1928]\label{thm:LocalStoilow}
  Let $f \colon \Sigma \to \Sigma'$ be a continuous, open and light map between surfaces.
  For each $x \in \Sigma$, there exists $k\in \N$, 
  a neighborhood $U$ of $z$, and homeomorphisms 
  $\psi \colon U \to \bD$ and $\phi \colon fU \to \bD$ for which $\psi(x)=0$, $\phi(f(x)) = 0$, and 
  the diagram
  \begin{align*}
    \xymatrix{
    U \ar[r]^{f|_U} \ar[d]_{\psi} & fU \ar[d]^\phi    \\
    \bD \ar[r]^{z\mapsto z^k}     & \bD 
                                    }
  \end{align*}
  commutes.
\end{theorem}

\begin{remark}
  For maps between oriented surfaces, we may alternatively state that there exist orientation-preserving homeomorphisms $\psi \colon U \to \bD$ and $\phi \colon fU \to \bD$ for which the local model for $f$ is $z\mapsto z^k$, if $f$ is orientation-preserving, and $z\mapsto \bar z^k$ if $f$ is orientation-reversing; see the proof of Theorem \ref{thm:LocalStoilow} in Section \ref{sec:last}.
\end{remark}

This local version of Sto\"ilow's discreteness theorem yields the following global factorization theorem. Indeed, by Theorem \ref{thm:LocalStoilow}, a continuous, light and open map $f\colon \Sigma \to S'$ induces a conformal structure on $\Sigma$ making it a Riemann surface $\tilde S$; we refer to
\cite[Lemma A.10]{Bonk-Meyer-book} for a short proof. Thus we may consider $f$ as a holomorphic map $f\colon \tilde S\to S'$. If the surface $\Sigma$ a priori carries a conformal structure, and we consider $\Sigma$ as a Riemann surface $S$, we may take the homeomorphism $h$ in the factorization to be the identity homeomorphism $\Sigma \to \Sigma$. 

\begin{theorem*}[Sto\"ilow's factorization theorem, 1938]
  Let $f \colon S \to S'$ be a continuous, open and light map between Riemann surfaces.
  Then there exists a Riemann surface $\tilde S$ and a homeomorphism $h \colon S \to \tilde S$ such
  that $f \circ h^{-1} \colon \tilde S \to S'$ is a holomorphic map.
\end{theorem*}

Recall that an
orientable topological surface carries a conformal structure. Indeed, by a classical theorem of Rad\'o,
every 2-manifold can be triangulated (see e.g.\ \cite[Theorem 8.3, p.\,60]{Moise} or \cite[Section II.8, p.\,105]{AhlforsSario})
and every triangulated orientable surface has a conformal structure
(see e.g.\ \cite[II.2.5E, Theorem, p.\,127]{AhlforsSario} or \cite[Section 2.2, pp.\,9--11]{Bers}).
In this way we recover the interpretation that \emph{a continuous, open and light map between orientable surfaces is a holomorphic map between Riemann surfaces}.

\bigskip

As the authors' interest for these theorems of Sto\"ilow stems from their role in the theory of quasiregular maps, we finish this introduction with a related remark. In the quasiconformal literature Sto\"ilow's theorem typically refers to the result that \emph{each quasiregular map $S\to S'$ between Riemann surfaces factors into a holomorphic map $S\to S'$ and a quasiconformal homeomorphism $S\to S$}. The proof of this result is analytic in its nature and based on the Beltrami equation. We refer to Astala, Iwaniec, and Martin \cite[Theorem 5.5.1]{Astala-Iwaniec-Martin-book} for a detailed discussion.

\bigskip \textbf{Acknowledgments.}
The authors would like to thank Mario Bonk for his encouragement
to write this expository article and for several helpful remarks. 
We are also grateful to the anonymous referee for many comments which considerably improved the exposition.

\section{Preliminaries}
\label{sec:Preli}

In the complex plane $\C$
we denote the unit ball as $\bD$ and the unit circle 
as $\bS^1$. For a path $c \colon [0,1] \to \C$ we set
$|c|$ to be the image set of $c$. An injective path
$\beta \colon [0,1] \to \C$ is called an \emph{arc}.
For a set $A \subset \C$ and a radius $r> 0$
we denote by $B(A,r)$ the set of points whose distance from $A$ is strictly
less than $r$.

A set is called \emph{totally disconnected} if its connected components are points.
A totally disconnected closed set in the plane has topological dimension 
zero. We recall this fact as the following lemma; see \cite[Section II.2, p.\,14]{Hurewicz-Wallman} for a proof.
\begin{lemma}\label{lemma:Topological}
  Let $C$ be a closed and totally disconnected set in $\C$.
  Then every point $x \in C$ has a neighborhood basis consisting
  of neighborhoods $U$ satisfying $C \cap \partial U = \emptyset$.
\end{lemma}

Let $\Omega$ be a planar domain and $f \colon \Omega \to \C$ be a 
continuous, open and light map. A domain $U \subset \Omega$ is a \emph{normal domain (for $f$)} if $U$ 
is compactly contained in $\Omega$ and $\partial f U = f \partial U$.
For each $x\in \Omega$ and $r>0$,
we denote by $U(x,f,r)$ the component of $f \inv B(f(x),r)$ containing the point $x$.
The following lemma shows that domains $U(x,f,r)$ are normal domains
of $x$ for all $r>0$ small enough.
The proof here is essentially identical to the proof in the case when $f$ is discrete;
see e.g.\ \cite[Lemma 5.1]{Vaisala} where the discreteness is used essentially only to guarantee
that Lemma \ref{lemma:Topological} holds.

\begin{lemma}
  \label{lemma:NormalNhoodExists}
  Let $\Omega$ be a planar domain, $f \colon \Omega \to \C$
  a continuous, open and light map, and let $x \in \Omega$. 
  Then there exists a radius $r_x > 0$
  such that for all $r \leq r_x$ the domain $U(x,f,r)$
  is a normal domain containing $x$
  and $f U(x,f,r) = B(f(x),r)$.
\end{lemma}
\begin{proof}
  Let $x \in \Omega$. By Lemma \ref{lemma:Topological} there
  exists a precompact 
  neighborhood $V$ of $x$ for which $\partial V \cap f \inv \{ f(x) \} = \emptyset$.
  Since $\partial V$ is compact
  and does not contain pre-images of $x$, we may fix
  $r_x >0$ for which $B(f(x),r_x)\cap f \partial V = \emptyset$.
  
  Let $0 < r < r_x$ and set $U \colonequals U(x,f,r)$.
  We show first that $U$ is a precompact domain contained in $V$. Since 
  \begin{align*}
    f(U \cap \partial V) 
    \subset fU \cap f \partial V
    \subset B(f(x),r) \cap f \partial V
    = \emptyset,
  \end{align*}
  we have $U \cap \partial V = \emptyset$. Since $U$ is a connected neighborhood of $x$ by definition, we therefore have that $U \subset V$. Thus
  $\overline{U}$ is compact as a closed subset
  of the compact set $\overline{V}$.

  We claim next that
  $\partial f U = f \partial U$.
  To show the inclusion $\partial fU \subset f \partial U$, we note first that, by precompactness
  of $U$ and continuity of $f$, we have that $\overline{fU} = f\overline{U}$. On the other hand,
  since $f$ is open and $U$ is a domain, both $U$ and $fU$ are open sets.
  Thus
  \begin{align*}
    \partial fU
    = \overline{fU} \setminus fU
    = (f\overline{U}) \setminus fU
    \subset f(\overline{U} \setminus U)
    = f \partial U.
  \end{align*}
  This proves the first inclusion.

  We show now the second inclusion $f \partial U \subset \partial f U$.
  Since $f$ is continuous, we have $f\partial U \subset f \overline{U} \subset \overline{fU}$
  and so it suffices to show that $fU \cap f\partial U = \emptyset$.
  Suppose towards contradiction that $fU \cap f\partial U \ne \emptyset$ and let $z \in fU \cap f \partial U$.
  Since $f$ is an open map and $U$ is a domain there exists a radius $s>0$ such that $B(z,s) \subset fU$.
  
  Fix $w \in (\partial U) \cap f \inv \{ z \}$. Since $f$ is continuous,
  $f \inv B(x,r)$ is open and thus intersects $U$ as $w \in (f\inv B(x,r)) \cap \partial U$.
  In particular, the $w$-component of $f \inv B(x,r)$, denoted $W$, intersects $U$.
  Now, as the open set $U$ is the $x$-component of
  $f \inv B(x,r) \supset f \inv B(z,s)$, it must in fact contain $W$, since the components
  of nested sets are also nested.
  This implies that $w \notin \partial U$, which is a contradiction.
  Thus $f \partial U \subset \partial f U$ and we conclude
  that $f \partial U = \partial f U$. Thus $U$ is a normal domain.

  Finally, we wish to show that $fU = B(f(x),r)$.
  To this end we first note that since $f$ is an open map and
  $U$ a domain, $fU$ is an open set in $B(f(x),r)$.
  Next we show that $fU$ must also be closed in $B(f(x),r)$.
  Suppose towards contradiction that there exists
  a point $z \in B(f(x),r) \cap \partial fU$. Since $\partial fU = f \partial U$,
  we may fix a point $w \in \partial U$ with $f(w) = z$.
  Now for $s < r - d(f(x),z)$ we see that the $w$-component $V$
  of $f \inv B(z,s)$ is a neighborhood of $w$ that intersects
  $U$. Since
  \begin{align*}
    V \subset
    f \inv B(z,s)
    \subset f\inv B(x,r)
  \end{align*}
  and both $U$ and $V$ are connected we must have by the definition of
  $U$ that $V \subset U$. But this is a contradiction since $w \in \partial U$,
  but $V$ is now a neighborhood of $w$ in $U$.
  Thus we conclude that $fU$ is both open and closed in $B(f(x),r)$, and so in particular
  $fU = B(f(x),r)$.
\end{proof}

We show next that, when restricted to a normal domain, a continuous, open and light map
is a proper map and furthermore the components of the pre-image 
of a compact set map surjectively to that compact set. For
continuous, open and discrete maps this claim is straightforward to 
prove, but for continuous, open and light maps the proof is more involved.
The proofs of the following two lemmas are a special case of a more general result of Whyburn; see
\cite[Theorem 7.4, p.\,147 and Theorem 7.5, p.\,148]{Whyburn-AT}.
\begin{lemma}\label{lemma:NormalProper}
  Let $U$ be a normal domain of a continuous, open and light
  map $f \colon \Omega \to \C$. Then the restriction
  $f|_U \colon U \to fU$ is a proper map.
\end{lemma}
\begin{proof}
  Let $K \subset fU$ be a compact set. Suppose towards contradiction that
  $ U \cap f \inv K$ is not compact. Then there
  exists a sequence $(x_j)$ in $U \cap f \inv K$ that has no
  subsequence converging in $U \cap f \inv K$.
  Since $\overline U$ is compact, the sequence $(x_j)$ does have a subsequence $(y_j)$ converging to a point $y_0 \in \overline U$. 
  Since $(f \inv K) \cap U$ is closed in $U$, 
  we conclude
  that $y_0 \in \partial U$. Now by the continuity of $f$ and closedness
  of $K$, we have that $f(y_0) \in K$. Since $y_0 \in \partial U$ and $U$ is a normal
  domain, this implies that $K \cap \partial fU \neq \emptyset$, which
  is a contradiction, since $K$ is a compact subset of the domain $fU$.
  We conclude that $f|_U$ is a proper map. 
\end{proof}

\begin{lemma}\label{lemma:NormalSurjectivity}
  Let $U$ be a normal domain of a continuous, open and light
  map $f \colon \Omega \to \C$. Then
  for any compact connected set $K \subset fU$,
  and each component $E$ of $(f|_U)\inv K$, the restriction 
  $f|_E \colon E \to K$ is surjective.
\end{lemma}
\begin{proof}
  Let $K \subset fU$ be a compact and connected set,
  and let $\tilde E = (f|_U)\inv K$. 
  We note that, since $f$ is open, 
  also the restriction $f|_{\tilde E} \colon \tilde E \to K$ is open.

  Suppose towards contradiction that there exists a component 
  $E \subset \tilde E$ and a point $p \in K$ for which $p \notin fE$. 
  Since $f|_U \colon U \to fU$
  is a proper map, the pre-image $P = U \cap f\inv \{ p \}$ is a compact set 
  which does not intersect $E$. 
  
  Since $\tilde E$ is a compact subset of the plane, it is especially
  a compact Hausdorff space. Thus the components of $\tilde E$ equal the
  quasicomponents of $\tilde E$, see e.g.\ \cite[Theorem 6.1.23]{Engelking};
  recall that the quasicomponent of a point is the intersection of all
  open and closed sets containing that point.
  Thus we may choose, for every $x \in P$, a separation
  $\{A_x,B_x\}$ of the set $\tilde E$, i.e.,\ two disjoint open sets 
  $A_x, B_x \subset U$
  covering $\tilde E$ for which $E \subset A_x$ and $x \in B_x$. 
  Since $P$ is compact, the cover $\{ B_x \colon x\in P\}$ has a finite subcover
  $\{ B_{x_1},\ldots, B_{x_n}\}$.

  For each $j=1,\ldots, n$, the open set $A_{x_j}$ contains the set $E$ and 
  satisfies $A_{x_j} \cap B_{x_j} = \emptyset$. Thus the intersection
  \begin{align*}
    A = \bigcap_{j=1}^n A_{x_j}
  \end{align*}
  is a neighborhood of $E$ satisfying
  $A \cap P = \emptyset$ and
  $E \cap \partial A = \emptyset$.
  Now $f A$ is an open
  set in $fU$ that contains $fE$ but does not contain $p$.
  Since $K$ is a connected set that intersects both $fA$ and
  its complement $\complement fA \ni p$,
  we have $(\partial fA ) \cap K \neq \emptyset$.

  Finally, on the one hand, by the openness of $f$ we have $\partial fA \subset f \partial A$.
  On the other hand, due to the
  definition of the sets $A_j$ and thus of $A$,
  $(\partial A) \cap f \inv K = \emptyset$. This is a contradiction and so
  the original claim holds.
\end{proof}

\section{Path lifting after Floyd and Sto\"ilow}
\label{sec:path-lifting}

In this section we discuss path-lifting, which
is one of the main tools in the forthcoming proofs.
The following theorem is, essentially, due to Sto\"ilow
\cite[pp.\,354--358]{Stoilow} and its idea was generalized by Floyd \cite{Floyd} to the setting
of compact metric spaces. We include here a version of Floyd's proof in the planar setting for the
sake of completeness of the exposition; note that the proof would go through
also in the setting of locally compact metric spaces.
\begin{theorem}\label{thm:PathLifting}
  Let $\Omega \subset \C$ be a planar domain and
  $f \colon \Omega \to \C$ a continuous, open and light map.
  Let $U$ be a normal domain in $\Omega$. Then for
  any path $\beta \colon [0,1] \to fU$ and any point 
  $x_0 \in U \cap f \inv \{ \beta(0) \}$ there exists a path 
  $\alpha \colon [0,1] \to U$ satisfying
  $\alpha(0) = x_0$ and $f \circ \alpha = \beta$.
\end{theorem}

We formulate a few lemmas for the proof of Theorem \ref{thm:PathLifting}.
For the statement of these auxiliary results we set
$\mathscr{C}(\Omega)$ to be the collection of all non-empty compact subsets
of $\Omega \subset \C$. This space is given the 
topology induced by the  \emph{Hausdorff distance $d_H$ of compact sets}, that
is, the distance $d_H(A,A')$ of compact sets $A,A'\subset \C$ is 
\begin{align*}
  d_H(A,A') = \inf
  \left\{
  \eps > 0 \mid 
  A \subset B(A',\eps),
  A' \subset B(A,\eps)
  \right\}.
\end{align*}

We show first that a continuous map $f\colon \Omega \to \C$ induces a continuous map $f_* \colon \mathscr{A} \mapsto \mathscr{fA}$.
For the properties of the Hausdorff distance and convergence, we refer to \cite[pp.\,70--77]{BridsonHaefliger} and
merely comment on the continuity of the induced map $f_*$.
\begin{lemma}\label{lemma:InducedMapIsContinuous}
  For a continuous mapping $f \colon \Omega \to \C$
  the induced map
  \begin{align*}
    f_* \colon \mathscr{C}(\Omega) \to \mathscr{C}(\C),
    \quad
    A\mapsto fA,
  \end{align*}
  is continuous in the topology induced by the Hausdorff distance.
\end{lemma}
\begin{proof}
  Let $A\subset \Omega$ be a compact set, denote $\dist(A,\partial \Omega) \equalscolon s_0$ and set $\varepsilon>0$. 
  Since $\Omega$ is locally compact, the continuous map $f$ is locally uniformly
  continuous and so there exists $\delta \in (0,\frac{s_0}{4})$ such that $d(f(x),f(y)) < \varepsilon$ for all
  $x,y \in B(A,\frac{s_0}{2})$ with
  $d(x,y) < \delta$. In particular, $fB(A,\delta) \subset B(fA, \varepsilon)$. 

  Let $A'\subset \Omega$ be a compact set for which $d_H(A,A') <\delta$.
  We show that $d_H(fA,fA') < \varepsilon$.
  Note first that since $fB(A,\delta) \subset B(fA, \varepsilon)$ and
  $A' \subset B(A,\delta)$, we have $fA' \subset B(fA,\varepsilon)$.
  On the other hand, for any $y_0 \in fA$ we can fix $x_0 \in A$ such that $f(x_0) = y_0$.
  Since $d_H(A,A') < \delta$ there now exists a point $z_0 \in A'$ with $d(z_0,x_0) < \delta$,
  and so by the aforementioned uniform continuity $d(f(z_0),y_0) < \varepsilon$, so in particular 
  $fA \subset B(fA',\varepsilon)$.
  Thus $d_H(fA,fA') < \varepsilon$ and the induced map $f_{*}$ is continuous.
\end{proof}

The following lemma is contained e.g.\ in \cite[pp.\,131,\,148]{Whyburn-AT},
but we include a modern proof for the reader's convenience.
\begin{lemma}\label{lemma:FloydDiameter}
  Let $\Omega \subset \C$ be a planar domain and 
  $f \colon \Omega \to \C$ a continuous, open and light map.
  Suppose $U$ is a normal domain in $\Omega$. Then for
  any $\eps > 0$ there exists a constant $\delta > 0$ having the property that,
  for any continuum 
  $C' \subset fU$ satisfying $\diam(C') < \delta$, the components of 
  $U \cap f \inv C'$ have diameter strictly less than $\eps$.
\end{lemma}
\begin{proof}
  Suppose there exists $\eps_0 > 0$ having the property that,
  for each $n \in \N$, there exists a continuum $C_n' \subset fU$ having diameter at most $1/n$ and a component $C_n$ of
  $U \cap f \inv C_n'$ having diameter at least $\eps_0$. Note that by Lemma \ref{lemma:NormalSurjectivity}
  $f(C_n) = C_n'$.

  Since both $\overline U$ and $f \overline U$ are compact,
  we may, after passing to subsequences if necessary, 
  assume that the sequences $(C_n')$ and $(C_n)$
  converge in the Hausdorff metric to a point $z \in f\overline{U}$ and to a continuum
  $C \subset \overline{U}$ with $\diam(C) \geq \eps_0$,
  respectively. Then, by the continuity of $f_*$, i.e.\ Lemma \ref{lemma:InducedMapIsContinuous},
  $f C = \{z\}$. Since $f$ is light, this is a contradiction. 
  Thus the claim is true.
\end{proof}

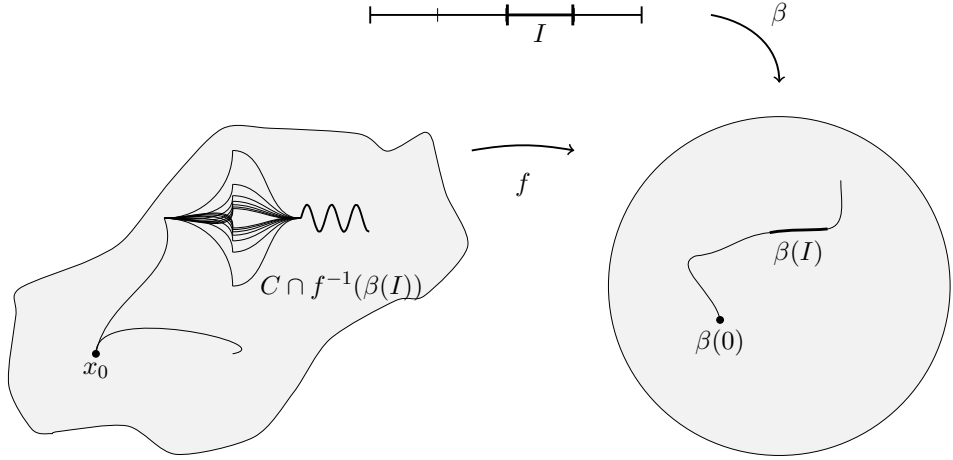
\begin{figure}
  \centering
  
  \resizebox{\textwidth}{!}{
    \begin{tikzpicture}

      \begin{scope}[scale=1.0]

        \begin{scope}[shift={(-4,0)}]

          \begin{scope}[scale=1.8,rotate=30]
            \draw [fill=black!5,smooth,domain=0:360] plot 
            (
            {2*cos( \x )+0.1*sin(34*\x)}, 
            {sin( \x ) + 0.2*cos(5* \x )}
            );
          \end{scope}
          
          \draw [fill] (-2,-1) circle [radius = 0.05];
          \node [below] at (-2,-1) {$x_0$};

          \draw (-2,-1) to [out=92,in=10] (0,-1);
          \draw (-2,-1) to [out=80,in=290] (-1,1);

          \foreach \i in {1,...,7} {
            \draw (-1,1) to [out=0, in=270+\i] (0,{1+1.0/\i});
            \draw (0,{1+1.0/\i}) to [out=0, in=180+\i] (1,1);

            \draw (-1,1) to [out=0, in=90-\i] (0,{1-1.0/\i});
            \draw (0,{1-1.0/\i}) to [out=0, in=180-\i] (1,1);
          }

          \begin{scope}[shift={(1,1)}]
            \draw[thick,domain=0:1,smooth] plot
            (
            {\x},
            {0.2*sin(1000*\x)}
            );
          \end{scope}
          
          \draw (1.6,-0.0) node { $C \cap f \inv ( \beta(I))$};

        \end{scope}
        
        \begin{scope}[shift={(4,0)}]
          \draw[fill=black!5] (0,0) circle [radius = 2.5];

          \coordinate (beta) at (210:1);
          \draw [fill] (beta) circle [radius = 0.05];
          \node[below] at (beta) {$\beta(0)$};
          
          \draw 
          (210:1) to [out=100,in=200]  
          (160:1.3) to [out=5,in=185]  
          (100:0.8) to [out=10,in=185]  
          (50:1.1) to [out=10,in=270]  
          (60:1.8);

          \draw [-,very thick] (100:0.8) to [out=10,in=185]  
          (50:1.1) ;
          \node[below] at ($(100:0.8)!0.5!(50:1.1)$) {$\beta(I)$};
        \end{scope}

        \begin{scope}[shift={(0,4)}]
          \draw [|-|, thick] (-2,0) -- (2,0);
          \draw [|-|, very thick] (0,0) -- (1,0);
          \foreach \i in {-1,0,1}
          {
            \draw (\i,0.1)--(\i,-0.1);
          }

          \node [below] at  (0.5,0)  { $I$ };
        \end{scope}
        
        \begin{scope}
          \draw[->, thick] (-0.5,2) to [out=10,in=170] (1,2);
          \draw (0.25,1.5) node {$f$};
          
          \draw[->, thick] (3,4) to [out=-10,in=90] (4,3);
          \draw (4,4) node {$\beta$};
        \end{scope}
        
      \end{scope}
      
    \end{tikzpicture}
  }
  \caption{An example of an $(f,\beta,2)$-prelift of a path.}
  \label{fig:PathLifting}
\end{figure}

With these preliminary results we can now turn to the proof of Theorem \ref{thm:PathLifting}.
To improve the clarity of the exposition we define some further auxiliary concepts.
The crucial idea from Floyd is, in the setting of Theorem \ref{thm:PathLifting},
to define for any subdivison $0 = t_0 < \ldots < t_k = 1$ of the unit interval a
continuum $C \subset U$ such that $fC = |\beta|$ and $C \cap f \inv ( \beta([t_j,t_{j+1}]))$ is a component
of $f \inv ( \beta([t_j,t_{j+1}]))$ for each $j = 0, \ldots, k-1$.

For the purpose
of lifting paths, the dyadic subdivisions turn out to be natural.
We formalize Floyd's idea by introducing the concept of a prelift.
\begin{definition}
  For a continuous, open and light mapping $f \colon \Omega \to \C$, a normal domain $U$ of $f$
  and a path $\beta \colon [a,b] \to fU$ we call a continuum $C = C(f,\beta,n) \subset U$
  an \emph{$(f,\beta,n)$-prelift of $\beta$ in $U$} if
  \begin{enumerate}
  \item for each $k = 1, \ldots 2^{n}$, $C \cap f \inv ( \beta [(k-1) 2^{-n}, k 2^{-n}] )$
    is a component of $f \inv (\beta [(k-1) 2^{-n}, k 2^{-n}])$, and
  \item for each $k = 0, \ldots 2^{n}$, $C \cap f \inv ( \beta(k 2^{-n}) ) \neq \emptyset$.
  \end{enumerate}
  We say that a prelift $C$ \emph{starts from $x_0 \in f \inv ( \beta(a) )$} if $x_0 \in C$.

  The \emph{restriction of a prelift} $C(f,\beta,n)$ to an interval $[c,d] \subset [a,b]$
  is just a subcontinuum of $C$ that is also an $(f,\beta|_{[c,d]},n)$-prelift. Note
  that an $(f,\beta,n)$ prelift is also a $(f,\beta,m)$-prelift for all $m \leq n$.
  
  Finally we denote by $w(C)$ the \emph{mesh} of a prelift $C = C(f,\beta,n)$, that is
  \begin{align*}
    w(C)
    \colonequals \max_{k = 1, \ldots, 2^{n}} \diam\left(C \cap f \inv (\beta [(k-1) 2^{-n}, k 2^{-n}])\right).
  \end{align*}
  Note that as an immediate corollary of Lemma \ref{lemma:FloydDiameter}
  and the uniform continuity of $\beta$ we see that
  $w(C(f,\beta,n)) \to 0$ as $n \to \infty$.
\end{definition}
A posteriori, for planar continuous, open and light mappings the definition of a prelift turns out
to be overly verbose; the prelifts for continuous, open
and discrete mappings are merely boundedly finite unions of path segments by
Rickman's path-lifting theorem in \cite{Rickman-Duke}.

\begin{lemma}\label{lemma:PreliftsExist}
  In the setting of Theorem \ref{thm:PathLifting}, for any $n \in \N$ and any given
  $x_0 \in U \cap f \inv \{ \beta(0) \}$, there exists an $(f,\beta,n)$-prelift of
  $\beta$ in $U$ starting from $x_0$.
\end{lemma}
\begin{proof}
  Fix $n \in \N$ and $x_0 \in U \cap f \inv \{ \beta(a) \}$.
  We set $C_0$ to be the component of $U \cap f \inv (\beta [0, 2^{-n}])$ containing $x_0$.
  By Lemma \ref{lemma:NormalSurjectivity} this component $C_0$ maps surjectively
  onto $\beta([0,2^{-n}])$.

  Suppose $C_{k-1}$ has been defined for some $k < 2^n$.
  By Lemma \ref{lemma:NormalSurjectivity}
  $C_{k-1}$ is mapped surjectively onto the set $\beta [ (k-1) 2^{-n}, k2^{-n}]$, which
  intersects $\beta [k 2^{-n}, (k+1) 2^{-n}]$ by the continuity of $\beta$.
  In particular $C_{k-1}$ contains a pre-image of the point $f \inv \{ \beta(k 2^{-n}) \}$,
  and so we may fix a component $C_k$ of
  $f \inv (\beta [ k 2^{-n}, (k+1)2^{-n}])$ such that $C_k$ and $C_{k-1}$ intersect.
  Again by
  Lemma \ref{lemma:NormalSurjectivity} this component $C_k$ maps surjectively
  onto $\beta([k 2^{-n}, (k+1) 2^{-n}])$.
  
  We now set $C \colonequals \cup_{j = 0}^{2^n - 1} C_j$ and note immediately that $C$
  is the required $(f,\beta,n)$-prelift of $\beta$.
\end{proof}

We are now ready to prove Theorem \ref{thm:PathLifting}. The idea of the proof relies
on taking a sequence of prelifts given by Lemma \ref{lemma:PreliftsExist} and
then constructing a subsequence of these prelifts that is, in a sense, converging on each
subinterval of $[0,1]$. Curious readers might be interested in
comparing this to the properties of ultralimits where one can, in a compact metric space,
choose a converging subsequence for each sequence in a consistent manner via a principal ultrafilter
in $\N$; see e.g.\ \cite[pp.\,77--80]{BridsonHaefliger}.

\begin{proof}[Proof of Theorem \ref{thm:PathLifting}.]
  Recall that by Lemma \ref{lemma:NormalProper} the restriction of a continuous, open and light
  mapping to a normal domain is a proper map.
  
  Fix a $(f,\beta,n)$-prelift $C_n$ for each $n \in \N$.
  We recursively construct a nested sequence of subsequences of $(C_n)$.
  First note that since $(C_n)$ is a sequence of continua in a
  a precompact domain $U$, it has a subsequence $(C_j^0)_j$ converging
  to a continuum $C \subset \overline{U}$ with respect to the Hausdorff metric.
  Furthermore since $f(C_n) \subset |\beta|$ with $|\beta| \cap \partial fU = \emptyset$,
  we have $C \subset U$.

  Suppose now that the subsequence $(C^{j-1}_m)_m$ has already been defined. We define the
  next subsequence $(C^{j}_m)_m$ in $2^{j}$ steps. First note that a restriction of the prelifts $C^{j-1}_m$ to the interval $[0,2^{-j}]$
  gives rise to a sequence of $(f, \beta|_{[0,2^{-j}]},j)$-prelifts. As before we
  may take a subsequence $(S_n^0)$ of $(C^{j}_m)_m$ such that restricted sequence converges.
  Now suppose that a subsequence $(S_n^k)_n$ of $(C^{j}_m)_m$ has been defined for $k < 2^{j}-1$. We set
  $(S_n^{k+1})$ to be the subsequence of $(S_n^k)$ such that the restrictions of these prelifts $S_n^k$
  to the interval $[k 2^{-j}, (k+1) 2^{-j}]$ form a converging sequence.
  Finally we set $(C^{j}_m)_m = (S^{2^{j}}_m)_m$.

  From these nested sequences $(C_n^j)_n$, $j \in \N$, we take the diagonal, i.e.\ we set
  $P_n = C_n^n$, and note that the sequence $(P_n)$ has the property that for any
  $j \in \N$ it is the subsequence of $(C_n^j)_n$ after possibly omitting finitely many elements.
  In particular, for any dyadic interval $\sigma \subset [0,1]$, the restrictions
  of $P_n$ to the interval $\sigma$ form a converging sequence in $U$ with respect to the Hausdorff metric.
  We show next that the limit of $P_n$, denoted $P$, is a path and thus a lift of $\beta$ under $f$.

  Fix $t_0 \in [0,1]$. Note that for any $m \in \N$ there are at most two intervals
  of the form $[k 2^{-m}, (k+1) 2^{-m}]$ containing $t_0$. Call the union of these two intervals
  $\sigma^m$.
  Since the restrictions of $P_n$ to any of the intervals $\sigma^m$ converge as $n \to \infty$, we thus
  see that with respect to the Hausdorff metric,
  \begin{align*}
    \left(P_n \cap f \inv ( \beta(\sigma^n) )\right)
    = \bigcap_{k=1}^n \left(P_n \cap f \inv ( \beta(\sigma^k) )\right)
    \to
    P \cap f \inv ( \beta(t_0) ).
  \end{align*}
  Furthermore, since the mesh of the prelifts $P_n$ tends to zero as $n \to \infty$, we note that also
  \begin{align*}
    \diam(P_n \cap f \inv ( \beta(\sigma^n) )) \to 0
    \quad
    \text{ as }
    \quad
    n \to \infty.
  \end{align*}
  Thus $P \cap f \inv ( \beta(t_0) )$
  is a point; call it $\alpha(t_0)$. We next note that since the mesh of the prelifts
  $P_n$ tends to zero, there exists for any $\epsilon > 0$ an integer $m_\epsilon \in \N$
  such that $d(\alpha(t_0), \alpha(t)) < \epsilon$ whenever $| t_0 - t| < 2^{-m_\epsilon}$.
  Thus the mapping $t \mapsto \alpha(t)$ is continuous.

  The path $\alpha \colon [0,1] \to U$ is now a path that maps onto $\beta$ under $f$,
  and $\alpha(0) = x_0$ by construction of the sequence $(C_n)$.
  Furthermore, since the prelifts were defined via pre-image components of pieces of
  $|\beta|$, we see that in fact $f \circ \alpha = \beta$ by Lemma 
  \ref{lemma:InducedMapIsContinuous}. Thus $\alpha$ is a lift of
  $\beta$ in $U$, and the claim holds true.
\end{proof}

We end this section with a proposition on the uniqueness of lifts into simply connected planar domains.
Note that this claim
clearly fails for maps between more general surfaces, even between spheres, and in higher dimensions. 
\begin{proposition}\label{prop:JCT-Basic}
  Let $\Omega \subset \C$ be a simply connected planar domain,
  $f \colon \Omega \to \C$ a continuous open map, and 
  let $\beta_1, \beta_2 \colon [0,1] \to \Omega$ be lifts
  of the same line segment $\alpha \colon [0,1] \to f\Omega$
  for which $\beta_1(0) = \beta_2(0)$ and
  $\beta_1(1) = \beta_2(1)$.
  Then $\beta_1 = \beta_2$.
\end{proposition}

Proposition \ref{prop:JCT-Basic}
is an almost immediate consequence of
the following version of the Jordan curve theorem;
for a proof we refer to \cite[Section 4]{Moise} or
\cite[Theorem 1.10, p.\,33]{Pommerenke}. Furthermore we emphasize here
that by the boundary of a domain $U$ compactly contained in a planar domain
$\Omega$ we mean the boundary relative to $\Omega$. 
\begin{theorem*}[The Jordan curve theorem]\label{thm:Jordan}
  Let $\Omega \subset \C$ be a simply connected planar domain and
  let $c \colon \bS^1 \to \Omega$ be an injective continuous map.
  Then $\Omega \setminus |c|$ consists of two domains, exactly one
  of which has a compact closure in $\Omega$.
  Both of these domains have the image of the curve $c$ as their boundary.
\end{theorem*}

Besides the Jordan curve theorem we need also the following
result stating that a line segment is not a boundary of a bounded domain in the plane. 
We formulate the statement of the lemma in slightly more general form for later use.
\begin{lemma}\label{lemma:NoArcBoundary}
  Let $\Omega$ be a precompact planar domain and $\beta \colon [0,1] \to \Omega$ a line segment.
  Suppose $r>0$ is a radius such that the closed balls
  $\overline{B}(\beta(0),r)$ and $\overline{B}(\beta(1),r)$ are disjoint and contained in $\Omega$.
  Then there exists no precompact domain $W \subset \Omega$ such that
  \begin{align*}
    \partial W
    \subset |\beta| \cup \overline{B}(\beta(0),r) \cup \overline{B}(\beta(1),r)
  \end{align*}
  and
  \begin{align*}
    \left(|\beta| \setminus (\overline{B}(\beta(0),r) \cup \overline{B}(\beta(1),r)) \right) \cap \partial W \neq \emptyset.
  \end{align*}
\end{lemma}
\begin{proof}
  Towards contradiction, suppose there exists such a precompact domain $W$.
  Let
  \begin{align*}
    z \in \left(|\beta| \setminus (\overline{B}(\beta(0),r) \cup \overline{B}(\beta(1),r)) \right) \cap \partial W
  \end{align*}
  and take a radius $s>0$ such that $\overline{B}(w,s)$ intersects neither of the closed balls
  $\overline{B}(\beta(0),r)$ and $\overline{B}(\beta(1),r)$.

  Now the line segment $\beta$ divides $B(w,s)$ into two domains, one of which must be contained
  in $W$ since $B(w,s) \cap \partial W$ is contained in the line segment $|\beta|$. Denote that domain by $V$.
  Let now $\alpha \colon [0,1) \to \Omega$ be a line segment, perpendicular to $\beta$, that starts from
  $w$, intersects $V$ and ends at the boundary of $\Omega$.
  Since $\alpha$ does not intersect either of the balls
  $\overline{B}(\beta(0),r)$ and $\overline{B}(\beta(1),r)$, we
  conclude that it intersects the boundary of $W$ at most once at $\alpha(0)$.
  Thus $\alpha(0,1) \subset W$ and so $W$ is not precompact as $\alpha(1) \notin W$.
\end{proof}

\begin{proof}[Proof of Proposition \ref{prop:JCT-Basic}]
  Suppose $\beta_1 \neq \beta_2$. Then there exists $t_0 \in [0,1]$ for which
  $\beta_1(t_0) \neq \beta_2(t_0)$.
  Let
  \begin{align*}
    a = \sup \{ t \in [0,t_0] \mid \beta_1(t) = \beta_2(t) \}
    \textrm{ and }
    b = \inf \{ t \in [t_0,1] \mid \beta_1(t) = \beta_2(t) \}.
  \end{align*}
  Then there exists an injective continuous map $c \colon \bS^1 \to \Omega$ for which
  $|c| = | \beta_1|_{[a,b]}  | \cup |\beta_2|_{[a,b]}|$, and so
  $|f\circ c| = \alpha[a,b]$.
  By the Jordan curve theorem, one of the components of $\Omega \setminus |c|$,
  say $U$, is a precompact subdomain of $\Omega$ with $\partial U = |c|$.

  Thus $fU$ is a precompact domain in $\C$ 
  and $\partial fU \subset |\alpha|$. This is a contradiction with 
  Lemma \ref{lemma:NoArcBoundary} since $\alpha$ is a line segment.
  The claim follows.
\end{proof}

\section{Proof of Theorem \ref{thm:light_to_discrete}}
\label{sec:FromLightToDiscrete}

We use now the path-lifting result, Theorem \ref{thm:PathLifting}, and the Jordan curve theorem
to prove Theorem \ref{thm:light_to_discrete}, that is, to show that a continuous, open and light planar map is discrete. We follow here an idea of Sto\"ilow that the discreteness of the mapping
follows from the finiteness of lifts of line segments. Sto\"ilow calls the following proposition, together with the existence of lifts, \emph{Th\'eor\`eme Fondamental} 
(\cite[p.\,361]{Stoilow}).

\begin{proposition}\label{prop:LiftBound}
  Let $\Omega$ be a planar domain and
  $f \colon \Omega \to \C$ a continuous, open and light map.
  Let $x \in \Omega$ and let $r>0$ be so small that 
  $U \colonequals U(x,f,r)$ is a 
  normal domain contained in a simply connected neighborhood of $x$ in $\Omega$.
  Suppose $\beta \colon [0,1] \to B(f(x),r)$
  is a ray $t \mapsto t z_0 + f(x)$, where $z_0 \in \bS^1$.
  Then there exists at most 
  finitely many lifts of $\beta$ in $U$.
\end{proposition}

With this proposition at our disposal, Theorem \ref{thm:light_to_discrete} follows immediately.
\begin{proof}[Proof of Theorem \ref{thm:light_to_discrete} assuming Proposition \ref{prop:LiftBound}]
  Let $x_0 \in \Omega$ be a point and fix a normal domain
  $U_0$ of $x_0$ contained in some simply connected neighborhood
  of $x$ in $\Omega$.
  By Theorem \ref{thm:PathLifting} a ray 
  in $fU_0$ starting from $f(x_0)$ has a lift starting from each pre-image
  of $f(x_0)$ in $U_0$.
  By Proposition \ref{prop:LiftBound}, such a ray 
  has only finitely many lifts in
  $U_0$. Thus the set $U_0 \cap f \inv \{ f(x_0) \}$ is finite and so $f$ is a discrete map.
\end{proof}

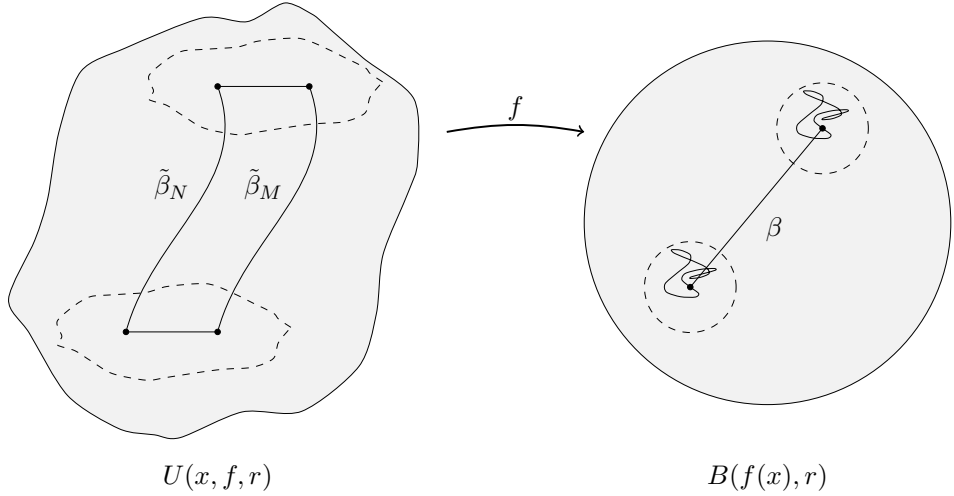
\begin{figure}
  \centering
  \resizebox{\textwidth}{!}{
    \begin{tikzpicture}

      \begin{scope}[scale=1.3]
        
        \begin{scope}[shift={(3,0)}]
          \draw[fill,black!5] (0,0) circle [radius = 2];
          \draw[] (0,0) circle [radius = 2];

          \coordinate (a) at (220:1.1);
          \coordinate (b) at (60:1.2);

          \foreach \point in {a,b}{
            \draw[fill] (\point) circle [radius=0.03];
            \draw[dashed] (\point) circle [radius=0.5];

            \begin{scope}[shift={(\point)},scale=0.15,rotate=90]
              \draw[smooth,domain=0:360] plot
              (
              {1+cos( \x )+sin(2*\x)},
              {sin( \x )+sin(4*\x)}
              );
            \end{scope}
          }
          
          \draw[-] (a) -- (b);
          \node[below right] at  ($(a)!0.5!(b)$) {$\beta$};

          \node at (0,-2.8) {$B(f(x),r)$};
        \end{scope}

        \draw[->, thick] (-0.5,1) to [out=10,in=170] (1,1);
        \node [below] at (0.25,1.5) {$f$};

        \begin{scope}[shift={(-3,0)}]
          
          \begin{scope}[rotate=60]
            \draw [fill=black!5,smooth,domain=0:360] plot 
            (
            {2.5*cos( \x )+0.1*sin(34*\x)}, 
            {2*sin( \x ) + 0.2*cos(5* \x )}
            );
          \end{scope}

          \coordinate (0) at (0.5,1.5);
          \coordinate (1) at (-0.5,-1.2);

          \foreach \point in {0,1} 
          {
            \begin{scope}[shift={(\point)},scale=0.5 ]
              \draw [dashed,smooth,domain=0:360] plot 
              (
              {2.5*cos( \x )+0.1*cos(34*\x)}, 
              {1*sin( \x ) + 0.1*cos(5* \x )}
              );

              \coordinate (N\point) at (-1,0);
              \coordinate (M\point) at (1,0);
              
              \draw[fill] (N\point) circle [radius=0.06];
              \draw[fill] (M\point) circle [radius=0.06];
              \draw (N\point)--(M\point);
            \end{scope}
          }
          
          \foreach \point in {N,M}
          {
            \draw (\point0) to [out=290, in=75] (\point1); 
            \node[above] at ($(\point0)!0.5!(\point1)$) {$\tilde\beta_\point$};
          }  

          \node at (0,-2.8) {$U(x,f,r)$};

        \end{scope}

      \end{scope}

    \end{tikzpicture}
  }
  \caption{Constructing a ``rectangular Jordan curve'' in the proof of Proposition
    \ref{prop:LiftBound}.}
  \label{fig:LtD-2}
\end{figure}

For the proof of Proposition \ref{prop:LiftBound} we require a few auxiliary results.
The first one is the following
technical lemma about the intersection properties of a sequence of sets
which will be applied to the images of a sequence of paths.
The proof requires no geometric properties of paths or planar topology;
this result is actually a version of the infinite Ramsey theorem for 
two colors; see e.g.\;\cite[Theorem 5, p.\,16]{GRS-Book}.
\begin{lemma}\label{lemma:Ramsey}
  Let $X$ be a set and $(A_n)$ be a sequence of subsets of $X$.
  Then there exists a subsequence $(C_n)$ of $(A_{n})$ 
  for which either
  \begin{enumerate}[(a)]
  \item all the sets $C_n$ are mutually disjoint, or
  \item any two of the sets $C_n$ intersect.
  \end{enumerate}
\end{lemma}
\begin{proof}
  We construct fist a subsequence $(B_n)$ of $(A_n)$
  with the property that for any $k \in \N$ either
  \begin{enumerate}[(i)]
  \item $B_j \cap B_k = \emptyset$ for all $j > k$, or
  \item $B_j \cap B_k \neq \emptyset$ for all $j > k$.
  \end{enumerate}
  We construct the sequence $(B_n)$ recursively together
  with a sequence of nested auxiliary subsequences $(B^k_n)_n$ of $(A_n)$.
  Set first $B_0 = A_0$. Now either $B_0$ intersects
  infinitely many elements of $(A_n)_{n \geq 1}$ or
  is disjoint from infinitely many elements of $(A_n)_{n \geq 1}$.
  Choose one such infinite collection and call it $(B^0_n)$.
  
  Suppose that $(B_0, \ldots, B_k)$ and $((B^0_n), \ldots, (B^k_n))$
  have been defined. We set $B_{k+1} = B^{k}_0$ and note that either
  $B_{k+1}$ intersects
  infinitely many elements of $(B^{k}_n)_{n\geq 1}$ or
  is disjoint from infinitely many elements of $(B^{k+1}_n)_{n \geq 1}$.
  Choose one such infinite collection and call it $(B^{k+1}_n)$.
  
  Since for all $k \in \N$ the sequence $(B^{k}_n)_n$ is a subsequence
  of $(B^{k-1}_n)_n$, we note that the required property holds for
  $(B_n)$, i.e.\ for any $k \in \N$ either (i) or (ii) holds.

  To complete the proof we note that there are two types of
  elements in the sequence $(B_n)$; those satisfying condition
  (i) and those satisfying condition (ii). Thus there must be infinitely
  many of elements of either type (i) or type (ii). Call one such collection
  $(C_n)$. Now either all elements of $(C_n)$ intersect all of the consequtive elements
  of the sequence
  or all elements of $(C_n)$ intersect none of the consequtive elements of the sequence.
  Thus we see that $(C_n)$ satisfies the condition in the statement of the lemma and
  so the proof is complete.
\end{proof}

Finally, we require the following lemma that allows us to extend pairs of disjoint arcs into a Jordan curve.
\begin{lemma}\label{lemma:JordanConstruction}
  Let $\Omega$ be a planar domain with $A,B \subset \Omega$
  domains such that $\overline{A} \cap \overline{B} = \emptyset$
  and $\overline{A}, \overline{B} \subset \Omega$. Suppose there
  exists two arcs $\alpha, \beta \colon [0,1] \to \Omega$ for which
  $\alpha(0), \beta(0) \in A$,
  $\alpha(1), \beta(1) \in B$ and 
  $|\alpha| \cap |\beta| = \emptyset$.
  Then there exists a Jordan curve $c \colon \bS^1 \to \Omega$ for which
  $|\alpha| \setminus (A \cup B) \subset |c|$,
  $|\beta| \setminus (A \cup B) \subset |c|$ and 
  $|c| \subset A \cup B \cup | \alpha| \cup |\beta|$.
\end{lemma}
\begin{proof}
  Planar domains are arc-connected, so there exists
  arcs $\gamma_A \colon [0,1] \to A$ and $\gamma_B \colon [0,1] \to B$ connecting
  the pairs of points $(\alpha(0), \beta(0))$ and $(\alpha(1),\beta(1))$, respectively.
  Fix
  \begin{align*}
    t_A^1
    \colonequals \sup \{ t \in [0,1] \mid \gamma_A(t) \in |\alpha| \},
    \quad t_A^2
    \colonequals \inf \{ t \in [t_A^1,1] \mid \gamma_A(t) \in |\beta| \}, \\
    t_B^1
    \colonequals \sup \{ t \in [0,1] \mid \gamma_B(t) \in |\alpha| \}, \text{ and }
    t_B^2
    \colonequals \inf \{ t \in [t_B^1,1] \mid \gamma_B(t) \in |\beta| \}.
  \end{align*}
  We note that the restrictions $\gamma_A |_{[t_A^1, t_A^2]}$ and $\gamma_B |_{[t_B^1, t_B^2]}$ are arcs
  that meet the paths $\alpha$ and $\beta$ only at their endpoints.
  Thus we may define the Jordan curve $c$ as a concatenation of these two restrictions together with
  the restrictions of $\alpha$ and $\beta$ to suitable subintervals.
\end{proof}

\begin{proof}[Proof of Proposition \ref{prop:LiftBound}]
  Suppose a ray $\beta\colon [0,1]\to fU$ has infinitely many mutually
  distinct lifts $\tilde\beta_n \colon [0,1] \to U$, $n \in \N$.
  By passing to a subsequence if necessary, we may assume that
  $\tilde\beta_n(0) \to x_0 \in \overline{U}$ and
  $\tilde\beta_n(1) \to y_0 \in \overline{U}$ as $n\to \infty$. 
  Since $U$ is a normal domain and $\beta(0),\beta(1) \in B(f(x),r)$, 
  we have, in fact, that $x_0, y_0 \in U$.

  By Lemma \ref{lemma:Ramsey} we may assume, by passing again to a subsequence if necessary, that 
  for the sequence $(\tilde \beta_n)$ either
  \begin{enumerate}[(a)]
  \item the images of the lifts $\tilde \beta_{n}$ are mutually disjoint, or
  \item any two of the images of the lifts $\tilde \beta_{n}$ intersect.
  \end{enumerate}
  
  Suppose first that (a) holds. 
  Let $s > 0$ be a radius for which  
  $V =U(x_0,f,s)$ and $V' =U(y_0,f,s)$ 
  are mutually disjoint normal domains
  of $x_0$ and $y_0$, respectively.
  Let $M, N \in \N$ be distinct indices for which 
  \begin{align*}
    \tilde \beta_M(0),\tilde \beta_N(0) \in V
    \quad \text{ and } \quad
    \tilde \beta_M(1),\tilde \beta_N(1) \in V'. 
  \end{align*}
  Then there exists by Lemma \ref{lemma:JordanConstruction} a Jordan curve $c \colon \bS^1 \to U$
  for which 
  \begin{align*}
    |c| \subset V \cup V' \cup |\tilde \beta_M| \cup |\tilde \beta_N |\quad \text{and}\quad |c| \not \subset V \cup V'.
  \end{align*}
  The image $|f\circ c|$ is a continuum
  contained in the ray $|\beta|$ and two mutually
  disjoint disks $B(f(x_0), s)$ and $B(f(y_0),s)$ located at the endpoints of $\beta$;
  see Figure \ref{fig:LtD-2}. 
  Let $K = |\beta| \cup \left( \bar B(f(x_0),s) \cup \bar B(f(y_0),s)\right)$.

  By the Jordan curve theorem
  the curve $c$ bounds a precompact domain $W$ in $U$, having boundary equal to $|c|$.
  Now, on the one hand, the boundary $\partial fW$ of $fW$ is contained in $|c| \subset K$.
  Also note that since $U = U(x,f,r)$, we in particular have that $fU = B(f(x),r)$ and so $fW$ is a precompact domain in a disk.
  On the other hand, the closure of $fW$ contains the segment $|\beta|\setminus \left( B(f(x_0),s)\cup B(f(y_0),s)\right)$.
  Thus, by connectedness of $fW$, we have $fW \supset B(f(x),r) \setminus K$. In particular, $fW$ is not precompact in $B(f(x),r)$.
  By Lemma \ref{lemma:NoArcBoundary} there can be no such domain $fW$.
  This is a contradiction, and the proof in case (a) is complete.

  Suppose now that (b) holds.
  By Proposition \ref{prop:JCT-Basic} one of the sets 
  \begin{align*}
    \{ \tilde \beta_j(1) \mid j \in \N \}
    \quad\text{ or }\quad
    \{ \tilde \beta_j(0) \mid j \in \N \}
  \end{align*}
  must be infinite. Thus we may assume, by changing the direction of the path $\beta$ and 
  by passing to a subsequence if necessary, that
  the endpoints $\tilde \beta_j(1)$ are all distinct.
  
  Let $\gamma \colon [0,1] \to B(f(y_0),r)$ be a ray for which 
  $\gamma(0) = f(y_0)$ and $|\gamma| \cap |\beta| = f(y_0)$;
  see Figure \ref{fig:LtD-1}. For each $n \in \N$,
  let $\tilde \gamma_n$ be a lift of $\gamma$ 
  starting from $\tilde \beta_n(1)$.
  We show next that all the lifts $\tilde \gamma_n$ are pair-wise disjoint.
  Towards contradiction suppose not and take $\tilde \gamma_k$ and $\tilde \gamma_j$
  such that $\tilde \gamma_k(s_0) = \tilde \gamma_j(s_0)$ for some $s_0 \in [0,1]$.
  Now by the property (b) the two lifts $\tilde \beta_k$ and $\tilde \beta_j$
  with corresponding indices also intersect;
  $\tilde \beta_k(t_0) = \tilde \beta_j(t_0)$ for some $t_0 \in [0,1]$.
  We may assume $t_0$ is the first such intersection point, i.e.\
  either $t_0 = 0$ or $\tilde \beta_k(t) \neq \tilde \beta_j(t)$ for all $t < t_0$.
  For $i = k,j$ we define the concatenations $\theta_i \colon [0,1] \to U$
  by setting
  \begin{align*}
    \theta_i(t)
    =
    \begin{cases}
      \tilde \beta_i (2t), & t \in [0,1/2] \\
      \tilde \gamma_i (2t-1), & t \in [1/2,1].
    \end{cases}
  \end{align*}
  Now applying Proposition \ref{prop:JCT-Basic} to the restrictions
  \begin{align*}
    \theta_k|_{[t_0/2,(s_0+1)/2]}
    \quad
    \text{ and }
    \quad
    \theta_j|_{[t_0/2,(s_0+1)/2]}
  \end{align*}
  gives rise to a contradiction and we conclude
  that all the lifts $\tilde \gamma_n$ are pair-wise disjoint.
  
  Now the argument of case (a) applies to the
  sequence $(\tilde \gamma_n)$, which is again a contradiction
  and the proof is complete.
\end{proof}

\begin{figure}
  \centering
  \resizebox{\textwidth}{!}{
    \begin{tikzpicture}

      \begin{scope}[scale=1.2]

        \begin{scope}[shift={(-3,0)}]
          
          \begin{scope}[rotate=15]
            
            \draw [fill=black!5,smooth,domain=0:360] plot 
            (
            {2.5*cos( \x )+0.1*sin(34*\x)}, 
            {2*sin( \x ) + 0.2*cos(5* \x )}
            );
            
            \coordinate (X) at (-0.5,-1.5);
            \draw[fill] (X) circle [radius =0.05];
            \node[below] at (X) {$x_0$};

            \draw (X) to [out=90,in=180] (-1,0);
            \node[left] at (-0.5,0.5) {$\tilde \beta$};
            
            \def\N{15}
            
            \draw (-1,0) to (1-2/\N,0);
            
            \foreach \i in {1,...,\N}
            {
              \draw {(1-2/\i,0)} to [out=10,in=270] (1-1/\i,1/\i);
              \draw[fill] (1-1/\i,1/\i) circle [radius = 0.02];
            }
            
            \foreach \i in {3,...,\N}
            {
              \draw[black!80] {(1-1/\i,1/\i)} to [out=50,in=200] (1-1/\i+0.2,1/\i+0.2);
              \draw[fill] (1-1/\i+0.2,1/\i+0.2) circle [radius = 0.02];
              \draw[fill] (1-1/\i,1/\i) circle [radius = 0.02];
            }

            \draw[fill] (1+0.2,0.2) circle [radius = 0.02];
            
            \draw[fill] (1,0) circle [radius = 0.02];
            \node[below] at (1,0) {$y_0$};

            \node at (1,-1) {$U(y_0,f,r_0)$};

            \node at (-1.5,2.2) {$U$};
            
            \begin{scope}[shift={(1.1,0.3)},scale=0.35,rotate=60]
              \draw [dashed,smooth,domain=0:360] plot 
              (
              {2.5*cos( \x )+0.1*sin(34*\x)}, 
              {2*sin( \x ) + 0.2*cos(5* \x )}
              );

            \end{scope}

          \end{scope}
          
        \end{scope}

        \begin{scope}[shift={(3,0)}]
          \draw[fill,black!5] (0,0) circle [radius = 2];
          \draw[] (0,0) circle [radius = 2];

          \coordinate (a) at (200:1.5);
          \coordinate (b) at (50:1);
          \coordinate (c) at (60:1.2);

          \draw[fill] (a) circle [radius=0.03];
          \draw[fill] (b) circle [radius=0.03];
          \draw[fill] (c) circle [radius=0.03];
          
          \draw[-] (a) -- (b);
          \draw[-] (b) -- (c);

          \node[below right] at  ($(a)!0.5!(b)$) {$\beta$};
          \node[right] at  ($(b)!0.5!(c)$) {$\gamma$};

          \draw[dashed] (50:1) circle [radius = 0.5];
        \end{scope}

        \draw[->, thick] (-0.5,1) to [out=10,in=170] (1,1);
        \node [below] at (0.25,1) {$f$};
        
      \end{scope}
      
    \end{tikzpicture}
  }
  \caption{Path having infinitely many lifts with a joint starting point
    in the proof of Proposition \ref{prop:LiftBound}.}
  \label{fig:LtD-1}
\end{figure}
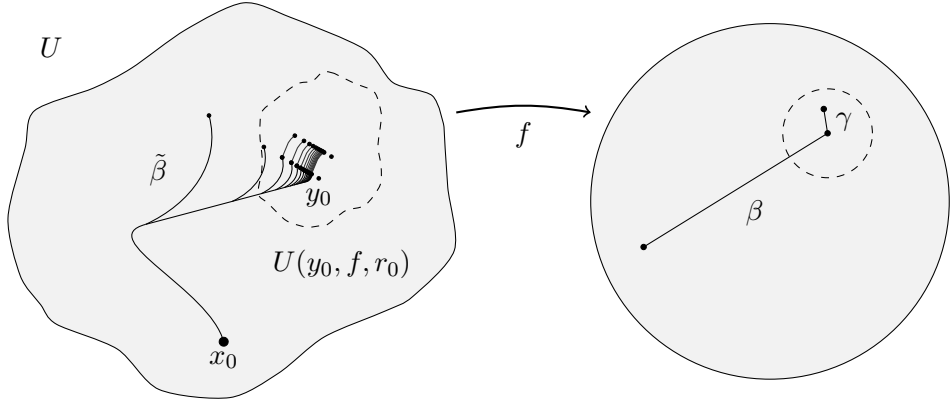

\section{Proofs of Theorems \ref{thm:discrete_branch} and \ref{thm:LocalStoilow}}
\label{sec:last}

Theorem \ref{thm:discrete_branch} is a direct consequence of 
the following proposition. For the statement, we say that a normal domain $U$ of a continuous, open and
discrete map $f\colon \Omega \to \C$ is a \emph{normal neighborhood of $x$ (with respect to $f$)}, if
$\overline{U} \cap f \inv \{f(x)\} = \{ x \}$.
The existence of such normal neighborhoods follows immediately by
noting that on the one hand for the discrete map $f$ and any
point $ x \in \Omega$, we have $d(x, f \inv \{ f(x) \} \setminus \{ x \} ) > 0$ and on the other hand that
for continuous, open and discrete mappings $\diam(U(x,f,r)) \to 0$ as $r \to 0$. The latter property in turn follows e.g.\ from
Lemma \ref{lemma:FloydDiameter}, but see also \cite[Lemma I.4.9]{RickmanBook}.

\begin{proposition}\label{prop:PlaneVaisala}
  Let $\Omega$ be a planar domain and let
  $f \colon \Omega \to \C$ be a continuous, open and discrete map.
  Let $x_0 \in \Omega$ and let $r>0$ be so small that $U_0 \colonequals U(x_0,f,r)$ is 
  a normal neighborhood of $x_0$ contained in a simply connected domain in $\Omega$. 
  Then $U_0 \cap B_f \subset \{ x_0 \}$.
\end{proposition}

Note that since $f \colon \Omega \to \C$ is a priori both continuous and open,
the branch set $B_f$ of $f$ is actually the set of points at which $f$ is not locally injective. Indeed, since a continuous and open bijection is a homeomorphism, we conclude that a continuous and open locally injective map is a local homeomorphism. In particular, if $f$ is locally injective at $x\in \Omega$, then $x$ has a neighborhood $U \subset \Omega$ for which $f|_U \colon U \to fU$ is a homeomorphism and $fU$ is an open neighborhood of $f(x)$. Thus $x\not \in B_f$.

\begin{proof}[Proof of Proposition \ref{prop:PlaneVaisala}]
  Suppose there exists
  $b \in (U_0 \cap B_f ) \setminus \{ x_0 \}$ and
  let $U\subset U_0$ be a normal neighborhood of $b$.
  Since $b \in B_f$, the map $f$ is not locally
  injective at $b$.
  Thus we may fix a point $y_0 \in fU$ for which 
  $\# (U \cap f \inv \{ y_0 \}) \geq 2$. Note that since $U$ is a normal neighborhood of
  $b$, $U \cap f \inv \{ f(b) \} = \{ b \}$ and so $y_0 \neq f(b)$; likewise $y_0 \neq f(x_0)$.

  Note that $fU$ is a planar disk. Thus we may fix two (piecewise linear) arcs 
  $\alpha \colon [0,1] \to fU$ and
  $\beta \colon [0,1]\to fU_0 \setminus \alpha(0,1]$ 
  satisfying
  $\alpha(0) = \beta(0) = y_0$, $\alpha(1) = f(b)$, and 
  $\beta(1) = f(x_0)$.
  Let also $z_1, z_2 \in U \cap f \inv \{ y_0 \}$, $z_1 \neq z_2$.
  By Theorem \ref{thm:PathLifting},
  there exists, for $i=1,2$, lifts 
  $\tilde \alpha_i \colon [0,1] \to U$ and 
  $\tilde \beta_i \colon [0,1] \to U_0$
  of $\alpha$ and $\beta$, respectively, satisfying $\tilde \alpha_i(0)=\tilde \beta_i(0) = z_i$ for $i=1,2$.
  Since $U_0$ and $U$ are normal neighborhoods of $x_0$ and $b$, respectively, 
  we have $\tilde \alpha_i(1) = b$, $\tilde \beta_i(1) = x_0$ for $i=1,2$.

  For $i =1,2$, let 
  $\tilde \gamma_i \colon [0,1] \to U$ be the path 
  \begin{align*}
    t \mapsto
    \begin{cases}
      \tilde \alpha_i (2t), & t \in [0,1/2] \\
      \tilde \beta_i (2t-1), & t \in [1/2,1].
    \end{cases}
  \end{align*}
  Since 
  $\tilde \gamma_1 (0) = \tilde \gamma_2(0)$,
  $\tilde \gamma_1 (1) = \tilde \gamma_2(1)$,
  and $U$ is contained in a simply connected domain, we have by 
  Proposition \ref{prop:JCT-Basic}
  that $\tilde \gamma_1 = \tilde \gamma_2$.
  This is a contradiction and the claim follows.
\end{proof}

We finish with a simple proof of Theorem \ref{thm:LocalStoilow} based on Theorem \ref{thm:discrete_branch} and a covering argument.
For the basic properties and terminology concerning covering maps and covering neighborhoods we refer to \cite{Hatcher}.

\begin{proof}[Proof of Theorem \ref{thm:LocalStoilow}]
  Let $z\in \Sigma$. Since the theorem posits the existence of some neighborhood, the claim 
  is local in nature and thus we may assume that $\Sigma$ and $\Sigma'$ are planar domains. 

  Let $r>0$ be so small that $U(z,f,r)$ is a 
  normal neighborhood of $z$ contained in some simply connected domain in $\Sigma$. 
  Denote 
  \begin{align*}
    U' = U(z,f,r) \setminus \{z\},
    B' = B(f(z),r) \setminus \{f(z)\}
    \text{ and }
    \bD' = \bD\setminus \{0\}.
  \end{align*}
  Since $f$ is discrete by Theorem \ref{thm:light_to_discrete}, 
  \begin{align*}
    U' \cap f \inv \{ y \} 
    \subset \overline{U} \cap f \inv \{ y \} 
  \end{align*}
  is finite for any $y \in fU$. Furthermore, by Proposition \ref{prop:PlaneVaisala}
  $f|_{U'} \colon U' \to B'$ is a local homeomorphism
  and since $U$ is a normal neighborhood,
  $f|_{U'}$ is also a proper map.
  Thus as a proper local homeomorphism
  $f|_{U'}$ is a covering map. Indeed, for any point $y_0 \in B'$ the pre-image $U' \cap f \inv\{ y_0 \}$ is a
  finite set $\{ x_1 , \ldots , x_k \}$ and we may fix disjoint open sets $U_j$
  such that the restriction $f|_{U_j} \colon U_j \to f U_j$, $j = 1, \ldots , k$, is a homeomorphism. 
  The open set $\cap_{j=1}^{k} f(U_j)$ is now a covering neighborhood
  of $y_0$, i.e.\ the restriction of $f$ to its pre-image components is a homeomorphism onto the set.

  As a finite cover of $B'$, the domain $U'$ is a topological punctured disk
  and we conclude that $U$ is a topological disk.
  Let $h_1 \colon U \to \bD$ and
  $h_2 \colon B(f(z),r) \to \bD$ be homeomorphisms 
  with $h_1(z) = 0$ and $h_2(f(z)) = 0$.
  Let also
  \begin{align*}
    g 
    \colonequals h_2 \circ f \circ (h_1|_{U'})\inv 
    \colon \bD' \to \bD'.
  \end{align*}
  Since $f|_{U'} \colon U' \to B'$ is a covering map, so is $g$.
  Thus the induced map $g_\ast \colon \pi_1(\bD') \to \pi_1(\bD')$
  is of the form  $m \mapsto km$ for some $k \in \Z \setminus \{0\}$.
  Note that $k\ne 0$, since the induced homomorphism $g_\ast$ is injective 
  by the homotopy lifting property of covering maps.

  We consider first the case $k > 0$.
  Let $\zeta_k \colon \bD'\to \bD'$ be the covering map $z\mapsto z^k$,
  and let $h' \colon \bD' \to \bD'$ be
  the lift of $g \colon \bD' \to \bD'$ under $\zeta_k$.
  Then $g = \zeta_k \circ h'$ and  
  $h'$ is a homeomorphism, since it is an injective covering map.
  The homeomorphism $h'$ extends to a homeomorphism $h \colon \bD \to \bD$ by the continuity of $\zeta_k$ and 
  hence $f|_U = h_2 \circ \zeta_k \circ h \circ h_1 \equalscolon \phi\inv \circ \zeta_k \circ \psi$.

  In the case $k < 0$ we repeat the previous argument by replacing the homeomorphism
  $h_2$ by the homeomorphism $k_2 \colon B(f(x),r) \to \bD$ defined by $x \mapsto \overline{h(z)}$.
\end{proof}


\def\cprime{$'$}

\end{document}